\documentclass[twoside,11pt]{amsart}

\usepackage[a4paper,hmargin={2.20cm,2.20cm},vmargin={2cm,2cm},heightrounded,marginparwidth=2.1cm, marginparsep=0.05cm]{geometry}
\usepackage[utf8]{inputenc}
\usepackage[T1]{fontenc}
\usepackage{lmodern}
\usepackage[shortlabels]{enumitem}
\usepackage[all,arc,cmtip]{xy}
\usepackage[mathscr]{euscript}
\usepackage{bbm}
\usepackage{amssymb}
\usepackage{amsthm}
\usepackage{mathtools}
\usepackage{cmll}
\usepackage{thmtools}
\usepackage{pigpen}
\usepackage{xcolor}
\usepackage{slashbox}
\usepackage[pdftex]{hyperref}

\usepackage{stmaryrd}
\SetSymbolFont{stmry}{bold}{U}{stmry}{m}{n}

\usepackage[nointegrals]{wasysym}
\SetSymbolFont{wasy}{bold}{U}{wasy}{m}{n}

\usepackage{etoolbox}
\AfterEndEnvironment{theorem}{\noindent\ignorespaces}

\DeclareMathAlphabet{\mathpzc}{OT1}{pzc}{m}{it}
\DeclareMathOperator{\yoneda}{\mathpzc{y}}

\xyoption{rotate}

\newtheorem{theorem}{{\bf Theorem}}[section]
\newtheorem{proposition}{\bf Proposition}[section]
\newtheorem{lemma}{\bf Lemma}[section]
\newtheorem{definition}{\bf Definition}[section]
 
\declaretheoremstyle[%
  spaceabove=-3pt,%
  spacebelow=7pt,%
  headfont=\normalfont\itshape,%
  postheadspace=1em,%
  qed=\qedsymbol%
]{mystyle} 
\declaretheorem[name={Proof},style=mystyle,unnumbered,
]{prf}

\newcommand{\categ}[1]{\mathsf{#1}}
\newcommand{\Cat}[1]{#1\text{-}\categ{Cat}}
\newcommand{\Rell}[1]{#1\text{-}\categ{Rel}}

\def\Set{\categ{Set}}
\def\BiRel{\categ{BiRel}}
\def\Top{\categ{Top}}
\def\BiTop{\categ{BiTop}}
\def\Pfn{\categ{Pfn}}
\def\App{\categ{App}}
\def\NA{\categ{NA}}  
\def\UApp{\categ{UApp}}
\def\Ord{\categ{Ord}}
\def\MultiOrd{\categ{MultiOrd}}
\def\Met{\categ{Met}}
\def\UltMet{\categ{UltMet}}
\def\ProbMet{\categ{ProbMet}}
\def\Equ{\categ{Equ}}
\def\PEqu{\categ{PEqu}}
\def\ALat{\categ{ALat}}
\def\Assm{\categ{Assm}}
\def\Mdst{\categ{Mdst}}
\def\one{\categ{1}}
\def\two{\categ{2}}
\def\V{\categ{V}}
\def\cX{\categ{X}}

\def\C{\mathcal{C}}

\def\x{\mathfrak{x}}
\def\w{\mathfrak{w}}

\def\U{\mathbbm{U}}
\def\I{\mathbbm{I}}
\newcommand{\T}{\mathbbm{T}}

\newcommand{\mL}{\mathbbm{L}}
\newcommand{\mM}{\mathbbm{M}}

\def\rel{{\longrightarrow\hspace*{-2.8ex}{\mapstochar}\hspace*{2.8ex}}}
\def\hom{{\rm hom}}

\newcommand{\X}{\mathcal{X}}
\newcommand{\Y}{\mathcal{Y}}
\newcommand{\Z}{\mathcal{Z}}

\newcommand{\mP}{\mathcal{P}} 

\title{On generalized equilogical spaces}

\author{Willian Ribeiro}
\address{CMUC, Department of Mathematics, University of Coimbra, 3001-501 Coimbra, Portugal}
\email{willian.ribeiro.vs@gmail.com}

\thanks{Research supported by Centro de Matem\'{a}tica da Universidade de Coimbra -- UID/MAT/00324/2013 and by the FCT PhD Grant PD/BD/128059/2016, funded by the Portuguese Government through FCT/MCTES and co-funded by the European Regional Development Fund through the Partnership Agreement PT2020.}



\begin{document}

\begin{abstract}
In this paper we carry the construction of equilogical spaces into an arbitrary category $\cX$ topological over $\Set$, introducing the category $\cX$-$\Equ$ of equilogical objects. Similar to what is done for the category $\Top$ of topological spaces and continuous functions, we study some features of the new category as (co)complete\-ness and regular (co-)well-poweredness, as well as the fact that, under some conditions, it is a quasitopos. We achieve these various properties of the category $\cX$-$\Equ$ by representing it as a category of partial equilogical objects, as a reflective subcategory of the exact completion $\cX_{_{{\rm ex}}}$, and as the regular completion $\cX_{_{{\rm reg}}}$. We finish with examples in the particular cases, amongst others, of ordered, metric, and approach spaces, which can all be described using the $\Cat{(\T,\V)}$ setting.
\end{abstract}

\maketitle

\section*{Introduction}

As a solution to remedy the problem of non-existence of general exponentials in $\Top$, Scott presents first in \cite{ANCDSER}, and later with his co-authors Bauer and Birkedal in \cite{MR2072989}, the category $\Equ$ of equilogical spaces. Formed by equipping topological $T_{_{0}}$-spaces with arbitrary equivalence relations, $\Equ$ contains $\Top_{_{0}}$ ($T_{_{0}}$-spaces and continuous functions) as a full subcategory and it is cartesian closed. This fact is directly proven by showing an equivalence with the category $\PEqu$ of partial equilogical spaces, which is formed by equipping algebraic lattices with partial (not necessarily reflexive) equivalence relations. Also in \cite{MR2072989}, equilogical spaces are presented as modest sets of assemblies over algebraic lattices, offering a model for dependent type theory. 

Contributing to the study of $\Equ$, a more general categorical framework, explaining why such (sub)catego\-ries are (locally) cartesian closed, was presented in \cite{MR1659606, MR1787592, MR1721095}. It turned out that $\Equ$ is related to the free exact completion $(\Top_{_{0}})_{_{{\rm ex}}}$ of $\Top_{_{0}}$ \cite{MR1358759, MR678508, MR1600009}. By the same token, suppressing the $T_{_{0}}$-separation condition on the topological spaces, the category $\Equ$ is a full reflective subcategory of the exact completion $\Top_{_{{\rm ex}}}$ of $\Top$. More, the reflector preserves products and special pullbacks, from where it is concluded in \cite{MR1659606} that $\Equ$ is locally cartesian closed, since $\Top_{_{{\rm ex}}}$ is so \cite[Theorem 4.1]{MR1659606}. It is shown in \cite{esfsgr} that $\Equ$ can be presented as the free regular completion of $\Top$ \cite{MR1358759, MR1600009}, and \cite{ecatmm} provides conditions for such regular completions to be quasitoposes. 

In this paper we start with a category $\cX$ and a topological functor $|$-$|\colon\cX\to\Set$, and, equipping each object $X$ of $\cX$ with an equivalence relation on its underlying set $|X|$, we define the category $\cX$-$\Equ$ of equilogical objects and their morphisms. Recovering the results for the particular case of $\Top$, $\cX$-$\Equ$ is (co)complete and regular (co-)well-powered. Under the hypothesis of pre-order-enrichment, we explore the concepts of separated and injective objects of $\cX$, leading us to the definition of a category $\cX$-$\PEqu$ of partial equilogical objects. In the presence of a separation condition, we proceed presenting $\cX$-$\Equ$ as modest sets of assemblies over injective objects; from that, we verify its properties of cartesian closedness and regularity. This is the subject of our first section. 

In Section \ref{sec3}, analogously to the case of $\Top$, we get similar results when considering the exact completion $\cX_{_{{\rm ex}}}$ and the regular completion $\cX_{_{{\rm reg}}}$ of $\cX$, culminating in the fact that $\cX$-$\Equ$ is a quasitopos, by the results of \cite{ecatmm}. To do so, we use a general approach to study weak cartesian closedness of topological categories (see \cite{CHRCCECT}).       

We finish with Section \ref{sec4}, where we briefly recall the $\Cat{(\T,\V)}$ setting, which was introduced in \cite{MR1957813} and further investigated in other papers \cite{MR1990036, MR2355608}, and study the case when $\cX=\Cat{(\T,\V)}$, for a suitable monad $\T$ and quantale $\V$, satisfying all conditions needed throughout the paper. Examples of such categories are $\Ord$ of preordered sets, $\Met$ of Lawvere generalized metric spaces \cite{MR1925933} and $\App$ of Lowen approach spaces \cite{MR1472024}, amongst others. Based on full embeddings among those categories, we place full embeddings among their categories of equilogical objects.

\section{The category of (partial) equilogical objects}\label{sec2}

Let $\cX$ be a category and $|$-$|\colon\cX\to\Set$ be a topological functor. In particular $\cX$ is complete, cocomplete, and $|$-$|$ preserves both limits and colimits. 

\begin{definition}\label{def11}{\rm The category $\cX$-$\Equ$ is defined as follows. \\
$\bullet$ The objects are structures $\X=\left\langle X,\equiv_{_{|X|}}\right\rangle$, where $X\in\cX$ and $\equiv_{_{|X|}}$ is an equivalence relation on the set $|X|$; they are called {\it equilogical objects of $\cX$}. \\
$\bullet$ A morphism from $\X=\left\langle X,\equiv_{_{|X|}}\right\rangle$ to $\Y=\left\langle Y,\equiv_{_{|Y|}}\right\rangle$ is the equivalence class of a morphism $f\colon X\to Y$ in $\cX$ such that $|f|$ is an {\it equivariant} map, i.e. $x\equiv_{_{|X|}}x'$ implies $|f|(x)\equiv_{_{|Y|}}|f|(x')$, for all $x,x'\in|X|$, with the equivalence relation on morphisms defined by
\begin{align*}f\equiv_{_{\X\to\Y}}g \ \Longleftrightarrow \ \forall \ x,x'\in|X|, \ (x\equiv_{_{|X|}}x' \ \Longrightarrow \ |f|(x)\equiv_{_{|Y|}}|g|(x')).\end{align*}}\end{definition}

One can see that $\equiv_{_{\X\to\Y}}$ is indeed an equivalence relation; reflexivity follows from the fact that the underlying maps are equivariant, symmetry and transitivity follow from the same properties for $\equiv_{_{|X|}}$ and $\equiv_{_{|Y|}}$.

Identity of $\X$ is given by $[1_{_{X}}]$ and composition of classes $[f]\colon\X\to\Y$ and $[g]\colon\Y\to\Z$ is given by $[g]\cdot[f]=[g\cdot f]$, which is well-defined. 

\begin{theorem}\label{theo1}$\cX$-$\Equ$ is complete, cocomplete, regular well-powered and regular co-well-powered.\end{theorem}

The proof of Theorem \ref{theo1} goes along the general lines of the proof of \cite[Theorem 3.10]{MR2072989}. Limits and colimits are computed in $\cX$ and their underlying sets are endowed with appropriate equivalence relations. The properties of regular well- and regular co-well-poweredness follow from the description of equalizers and coequalizers in $\cX$-$\Equ$. 

In general $\cX$-$\Equ$ is neither well-powered nor co-well-powered, as observed in \cite{MR2072989} for topological spaces. A morphism $[m]\colon\X\to\Y$ is a monomorphism in $\cX$-$\Equ$ if, and only if, 
\begin{align*}x\equiv_{_{|X|}}x' \ \Longleftrightarrow \ |m|(x)\equiv_{_{|Y|}}|m|(x'), \ \forall \ x,x'\in|X|.\end{align*}
A morphism $[f]\colon\X\to\Y$ is an epimorphism in $\cX$-$\Equ$ if, and only if, 
\begin{align*}y\equiv_{_{|Y|}}y' \ \Longleftrightarrow \ \exists \ x,x'\in|X|; \ x\equiv_{_{|X|}}x' \ \with \ y\equiv_{_{|Y|}}|f|(x)\equiv_{_{|Y|}}|f|(x')\equiv_{_{|Y|}}y'.\end{align*}

Having the {\it embedding} and the {\it extension} theorems configured for powersets \cite[Theorems 3.6, 3.7]{MR2072989}, according to the authors, Scott has pointed out that those results in fact hold more generally to continuous lattices. Powersets can be generalized to algebraic lattices, and it is explained that {\it ``The reason for considering algebraic lattices is that the lattice of continuous functions between powerset spaces is not usually a powerset space, but it is an algebraic lattice. And this extends to all algebraic lattices.''}, culminating in the well known fact that the category $\ALat$ of algebraic lattices and Scott-continuous functions is cartesian closed \cite[Chapter II, Theorem 2.10]{MR614752}. 

Algebraic lattices are in particular continuous lattices, therefore injective objects in $\Top_{_{0}}$. We show below that these -- injectivity and separation -- are the crucial properties in order to extend the arguments of \cite{MR2072989}. Next we assume that \\[0.4cm]
{\bf (a)} {\it $\cX$ is a pre-order enriched category.}

\begin{definition}\label{def201}{\rm An object $X$ of $\cX$ is said to be {\it separated} if, for each morphisms $f,g\colon Y\to X$ in $\cX$, whenever $f\simeq g$ ($f\leq g$ and $g\leq f$), then $f=g$.}\end{definition}

Hence an object $X$ is separated if, for each object $Y$, the pre-ordered set of morphisms $\cX(Y,X)$ is anti-symmetric. One can check that this is equivalent to the pre-ordered set $\cX(\one,X)$ to be anti-symmetric, where $\one=L\{*\}$, with $L\colon\Set\to\cX$ the left adjoint of $|$-$|$. 

The full subcategory $\cX_{_{{\rm sep}}}$ of separated objects is replete and closed under mono-sources. Since $({\rm RegEpi},\mathbb{M})$ is a factorization system for sources in the topological category $\cX$, where $\mathbb{M}$ stands for the class of mono-sources \cite[Proposition 21.14]{MR1051419}, closure under mono-sources then implies that $\cX_{_{{\rm sep}}}$ is regular epi-reflective in $\cX$ \cite[II-Proposition 5.10.1]{MR3307673}. 

We will consider (pseudo-)injective objects of $\cX$ with respect to $|$-$|$-initial morphisms. Hence, denoting by $\cX_{_{{\rm inj}}}$ the full subcategory on the injectives, $Z\in\cX_{_{{\rm inj}}}$ if, and only if, for each $|\text{-}|$-initial morphism $y\colon X\to Y$ and morphism $f\colon X\to Z$, there exists a morphism $\hat{f}\colon Y\to Z$ such that $\hat{f}\cdot y\simeq f$;
\begin{align*}\xymatrix{X \ar[rr]^{y} \ar[rd]_{f} \ar@{}[drr]^(0.45){\simeq} & & Y \ar[ld]^{\hat{f}} \\ & Z & }\end{align*}
$\hat{f}$ is called an {\it extension} of $f$ along $y$; if $Z$ is separated, then $\hat{f}\cdot y=f$. Moreover, assume that \\[0.4cm]
{\bf (b)} {\it for each $X,Y\in\cX$, $x,x'\in\cX(\one,X)$ and $|$-$|$-initial $f\in\cX(X,Y)$, 
\begin{align*}f\cdot x\simeq f\cdot x' \ \Longrightarrow \ x\simeq x'.\end{align*}}
\noindent Thus, for $X$ separated, if $f$ is $|$-$|$-initial, then $f$ is an embedding (regular monomorphism, which in our case is equivalent to $|$-$|$-initial with $|f|$ an injective map); hence restricting ourselves to the separated objects, injectivity with respect to $|$-$|$-initial morphisms coincides with injectivity with respect to embeddings. 

\begin{definition}\label{def21}{\rm The category $\cX$-$\PEqu$ of {\it partial equilogical objects of $\cX$} consists of: \\
$\bullet$ objects are structures $\X=\left\langle X,\equiv_{_{|X|}}\right\rangle$, where $X\in\cX_{_{{\rm inj}}}$ and $\equiv_{_{|X|}}$ is a partial (not necessarily reflexive) equivalence relation on the set $|X|$; \\
$\bullet$ a morphism from $\left\langle X,\equiv_{_{|X|}}\right\rangle$ to $\left\langle Y,\equiv_{_{|Y|}}\right\rangle$ is the equivalence class of an $\cX$-morphism $f\colon X\to Y$ such that $|f|$ is an equivariant map, with the equivalence relation on morphisms as in Definition \ref{def11}.}\end{definition}

In order to verify an equivalence between the categories of equilogical and partial equilogical objects, we will restrict ourselves to the separated objects, so we consider now that the objects in the structures of Definitions \ref{def11} and \ref{def21} are all separated, and denote the resulting categories by $\cX$-$\Equ_{_{{\rm sep}}}$ and $\cX$-$\PEqu_{_{{\rm sep}}}$, respectively. We also assume that \\[0.4cm]
{\bf (c)} {\it $\cX$ has enough injectives, meaning that, for each $X\in\cX$, there exists an $|$-$|$-initial morphism $\yoneda_{_{X}}\colon X\to\hat{X}$, with $\hat{X}\in\cX_{_{{\rm inj}}}$, and, if $X$ is separated, so is $\hat{X}$.} \\[0.4cm]
When $X$ is separated, as we have seen before, $\yoneda_{_{X}}$ is an embedding.

\begin{theorem}\label{teo21}$\cX$-$\Equ_{_{{\rm sep}}}$ and $\cX$-$\PEqu_{_{{\rm sep}}}$ are equivalent.\end{theorem}

\begin{prf}{\rm As in the proof of \cite[Theorem 3.12]{MR2072989}, a functor $R\colon\cX$-$\PEqu_{_{{\rm sep}}}\to\cX$-$\Equ_{_{{\rm sep}}}$ is defined taking each separated partial equilogical object $\X$ to $R\X=\left\langle RX,\equiv_{_{|RX|}}\right\rangle$, where $|RX|=\{x\in |X| \ | \ x\equiv_{_{|X|}}x\}$ and $\equiv_{_{|RX|}}$ is the restriction of $\equiv_{_{|X|}}$. For a morphism $[f]\colon\X\to\Y$, $|f|(|RX|)\subseteq|RY|$, so we take the (co)restriction $\overline{|f|}\colon|RX|\to|RY|$, lift to an $\cX$-morphism $\overline{f}\colon RX\to RY$ and set $R[f]=[\overline{f}]$.

To prove that $R$ is faithful one only needs to observe that, for elements $x,x'\in X$, if $x\equiv_{_{|X|}}x'$, then $x'\equiv_{_{|X|}}x$, and consequently $x\equiv_{_{|X|}}x$ and $x'\equiv_{_{|X|}}x'$, whence $x,x'\in|RX|$; and to prove that $R$ is full one uses the injectivity of $Y$, providing an extension $\hat{f}\colon X\to Y$ of $i_{_{RY}}\cdot f$ along $i_{_{RX}}$. 

Finally, for essential surjectivity let $\X=\left\langle X,\equiv_{_{|X|}}\right\rangle\in\cX$-$\Equ_{_{{\rm sep}}}$ and consider the embedding $\yoneda_{_{X}}\colon X\to\hat{X}$, $\hat{X}\in\cX_{_{{\rm sep,inj}}}$. Endow $|\hat{X}|$ with the following partial equivalence relation
\begin{align*}\varphi\equiv_{_{|\hat{X}|}}\psi\Longleftrightarrow \ \exists \ x,x'\in|X|; \ \varphi=|\yoneda_{_{X}}|(x),\psi=|\yoneda_{_{X}}|(x') \ \with \ x\equiv_{_{|X|}}x',\end{align*}
that is, two elements of $|\hat{X}|$ are related if, and only if, they are the images by $|\yoneda_{_{X}}|$ of elements that are related in $|X|$. The sets $|R\hat{X}|$ and $|X|$ are in bijection; using the $|$-$|$-initiality of $\yoneda_{_{X}}$ and $i_{_{R\hat{X}}}$, this bijection proves to be an isomorphism in $\cX$, and consequently in $\cX$-$\Equ_{_{{\rm sep}}}$, by the definition of $\equiv_{_{|\hat{X}|}}$, so $R\left\langle \hat{X},\equiv_{_{|\hat{X}|}}\right\rangle\cong\X$.}\end{prf}

For our next result we must assume that \\[0.4cm]
{\bf (d)} {\it every injective object of $\cX$ is exponentiable.} \\[0.4cm] 
Binary products and exponentials of injective objects are again injective, so $\cX_{_{{\rm inj}}}$ is a cartesian closed subcategory of $\cX$. We also assume that \\[0.4cm]
{\bf (e)} {\it the reflector from $\cX$ to $\cX_{_{{\rm sep}}}$ preserves finite products;} \\[0.4cm]
whence the exponential of separated objects, when it exists, is again separated \cite{MR0296126,MR785032}.

\begin{theorem}\label{teo22}$\cX$-$\PEqu_{_{{\rm sep}}}$ is cartesian closed.\end{theorem}

\begin{proof}{\rm Let $\X=\left\langle X,\equiv_{_{|X|}}\right\rangle$ and $\Y=\left\langle Y,\equiv_{_{|Y|}}\right\rangle$ be partial equilogical separated objects. We build the exponential $Y^{X}$ in $\cX_{_{{\rm sep,inj}}}$ and endow $|Y^{X}|$ with the partial equivalence relation: $\alpha\equiv_{_{|Y^{X}|}}\beta$ if, and only if, for all $x,x'\in X$,
\begin{align*} x\equiv_{_{|X|}}x' \ \Longrightarrow \ \alpha(x)=|{\rm ev}|(\alpha,x)\equiv_{_{|Y|}}|{\rm ev}|(\beta,x')=\beta(x'),\end{align*}
for each $\alpha,\beta\in|Y^{X}|$, where ${\rm ev}\colon Y^{X}\times X\to Y$ is the evaluation morphism in $\cX$. Then $\Y^\X=\left\langle Y^{X},\equiv_{_{|Y^{X}|}}\right\rangle\in\cX$-$\PEqu_{_{{\rm sep}}}$ and $|{\rm ev}|\colon|Y^{X}|\times|X|\to|Y|$ is equivariant, so $[{\rm ev}]\colon\Y^\X\times\X\to\Y$ is a valid morphism in $\cX$-$\PEqu_{_{{\rm sep}}}$. More, $[{\rm ev}]$ satisfies the universal property: for each morphism $[f]\colon\Z\times\X\to\Y$, $\Z=\left\langle Z,\equiv_{_{|Z|}}\right\rangle\in\cX$-$\PEqu_{_{{\rm sep}}}$, there exists a unique $[\overline{f}]\colon\Z\to\Y^{\X}$ commuting the diagram below.
\begin{align*}\xymatrix{\Y^{\X}\times\X \ar[rrr]^{[{\rm ev}]} & & & \Y \\ & & & \\ \Z\times\X \ar[uurrr]_{[f]} \ar[uu]^{[\overline{f}]\times1_{_{\X}}} & & & }\end{align*}
The morphism $\overline{f}\colon Z\to Y^{X}$ is the transpose of $f\colon Z\times X\to Y$, so that 
\begin{align*}{\rm ev}\cdot(\overline{f}\times 1_{_{X}})=f,\end{align*} 
and $[\overline{f}]$ is indeed unique, for if $[f']\colon\Z\to\Y^{\X}$ is such that $[{\rm ev}\cdot(f'\times1_{_{X}})]=[f]$, then for each $z\equiv_{_{|Z|}}z'$ in $Z$ and $x\equiv_{_{|X|}}x'$ in $X$, 
\begin{align*}\begin{array}{rl}\overline{f}(z)(x)={\rm ev}\cdot(\overline{f}\times1_{_{X}})(z,x)=f(z,x)\equiv_{_{|Y|}}f(z',x') \hspace{2.5cm} \\ \hspace{2.5cm} \equiv_{_{|Y|}}{\rm ev}\cdot(f'\times1_{_{X}})(z',x')=f'(z')(x'),\end{array}\end{align*} 
hence $\overline{f}(z)\equiv_{_{|Y^{X}|}}f'(z')$, i.e. $[\overline{f}]=[f']$.}\end{proof}

Therefore, by Theorem \ref{teo21}, $\cX$-$\Equ_{_{{\rm sep}}}$ is cartesian closed. We remark that the proof of Theorem \ref{teo22} also applies to $\cX$-$\PEqu$ without separation.

To finish this section we discuss the presentation of equilogical objects as modest sets of assemblies, following what is done in \cite[Section 4]{MR2072989}.

\begin{definition}\label{def22}{\rm The category of {\it assemblies} $\Assm(\cX_{_{{\rm inj}}})$ over injective objects of $\cX$ consists of the following data: objects are triples $(A,X,E_{_{A}})$, where $A$ is a set, $X\in\cX_{_{{\rm inj}}}$, and $E_{_{A}}\colon A\to\mP|X|$ is a function such that $E_{_{A}}(a)\neq\emptyset$, for each $a\in A$, with $\mP|X|$ the powerset of $|X|$. The elements of $E_{_{A}}(a)$ are called {\it realizers} for $a$. A morphism between assemblies $(A,X,E_{_{A}})$ and $(B,Y,E_{_{B}})$ is a map $f\colon A\to B$ for which there exists a morphism $g\colon X\to Y$ in $\cX$ such that $|g|(E_{_{A}}(a))\subseteq E_{_{B}}(f(a))$; $g$ is said to be a {\it realizer} for $f$, and we say that $|g|$ {\it tracks} $f$.}\end{definition}

\begin{definition}\label{def23}{\rm An object $(A,X,E_{_{A}})\in\Assm(\cX_{_{{\rm inj}}})$ is called a {\it modest set} if, for all $a,a'\in A$, $a\neq a'$ implies $E_{_{A}}(a)\cap E_{_{A}}(a')=\emptyset$. The full subcategory of the assemblies that are modest sets is denoted by $\Mdst(\cX_{_{{\rm inj}}})$.}\end{definition}

With these definitions, we get the same properties as those for the particular case of topological spaces, which we highlight in the following items, omitting some of the proofs that follow directly from the ones in \cite{MR2072989}. \vspace{0.4cm}

\noindent(1) {\it $\Mdst(\cX_{_{{\rm inj}}})$ and $\Assm(\cX_{_{{\rm inj}}})$ have finite limits and the inclusion from modest sets to assemblies preserves them.} \vspace{0.4cm}

\noindent(2) {\it $\Mdst(\cX_{_{{\rm inj}}})$ and $\Assm(\cX_{_{{\rm inj}}})$ are cartesian closed and $\Mdst(\cX_{_{{\rm inj}}})\to\Assm(\cX_{_{{\rm inj}}})$ preserves exponentials.} For $(A,X,E_{_{A}})$ and $(B,Y,E_{_{B}})$ in $\Assm(\cX_{_{{\rm inj}}})$, the exponential is $(C,Y^{X},E_{_{C}})$, where $C=\{f\colon A\to B \ | \ \exists \ g\colon X\to Y\in\cX \ \text{realizer for} \ f\}$, and $E_{_{C}}(f)=\{\alpha\in|Y^{X}| \ | \ \text{$\alpha$ tracks $f$}\}$; here, for simplicity, we denote also by $\alpha$ the map from $|X|$ to $|Y|$, given by $x\mapsto|{\rm ev}|(\alpha,x)$, for each $x\in|X|$, where ${\rm ev}\colon Y^{X}\times X\to Y$ is the evaluation map. If $(B,Y,E_{_{B}})$ is a modest set, then so is $(C,Y^{X},E_{_{C}})$, for if $f,f'\colon A\to B$ are tracked by $\alpha\in|Y^{X}|$, then for each $a\in A$, take $x\in E_{_{A}}(a)\neq\emptyset$, then $\alpha(x)\in E_{_{B}}(f(a))\cap E_{_{B}}(f'(a))\neq\emptyset$, whence $f(a)=f'(a)$ and then $f=f'$. \vspace{0.4cm}

\noindent(3) {\it $\Mdst(\cX_{_{{\rm inj}}})$ is a reflective subcategory of $\Assm(\cX_{_{{\rm inj}}})$.} \vspace{0.4cm}

\noindent(4) {\it The regular subobjects of $(A,X,E_{_{A}})$ in $\Assm(\cX_{_{{\rm inj}}})$, or in $\Mdst(\cX_{_{{\rm inj}}})$, are in bijective correspondence with the powerset of $A$.} \vspace{0.4cm}

\noindent(5) {\it $\Mdst(\cX_{_{{\rm inj}}})$ and $\Assm(\cX_{_{{\rm inj}}})$ are regular categories.} 

\begin{theorem}\label{teo23}$\cX$-$\PEqu$ and $\Mdst(\cX_{_{{\rm inj}}})$ are equivalent.\end{theorem}

\begin{prf}{\rm Define the functor $F\colon\Mdst(\cX_{_{{\rm inj}}})\to\cX$-$\PEqu$ assigning to $(A,X,E_{_{A}})$ the object 
$\left\langle X,\equiv_{_{|X|}}\right\rangle$, where 
\begin{align*}x\equiv_{_{|X|}}x' \ \Longleftrightarrow \ \exists \ a\in A; \ x,x'\in E_{_{A}}(a),\end{align*} 
and on morphisms $F$ assigns to each $f\colon(A,X,E_{_{A}})\to(B,Y,E_{_{B}})$ the class of a realizer $g\colon X\to Y$ for $f$; $\equiv_{_{|X|}}$ is indeed an equivalence relation and two realizers for $f$ are in the same equivalence class, so $F$ is well-defined.

Faithfulness of $F$ follows from the observation in item (2) above: two maps tracked by the same realizer must be equal. To see that $F$ is full, take a morphism $[g]\colon F(A,X,E_{_{A}})\to F(B,Y,E_{_{B}})$ in $\cX$-$\PEqu$. For each $a\in A$, let $x\in E_{_{A}}(a)\neq\emptyset$; then $x\equiv_{_{|X|}}x$ and so $|g|(x)\equiv_{_{|Y|}}|g|(x)$, that is, there exists $b\in B$ such that $|g|(x)\in E_{_{B}}(b)$, whence we set $f(a)=b$; this $b$ is uniquely determined since $(B,Y,E_{_{B}})$ is a modest set, therefore we have a map $f\colon A\to B$, which, by definition, has $g$ as a realizer.

Now let $\left\langle X,\equiv_{_{|X|}}\right\rangle$ be a partial equilogical object and define $(A,X,E_{_{A}})$ by $A=\{x\in|X| \ | \ x\equiv_{_{|X|}}x\}/\hspace{-0.11cm}\equiv_{_{|X|}}$ and $E_{_{A}}([x])=[x]\subseteq\mP|X|$. Hence $F(A,X,E_{_{A}})=\left\langle X,\equiv_{_{|X|}}\right\rangle$ and $F$ is essentially surjective.}\end{prf}  

The same argument can be repeated replacing $\cX_{_{{\rm inj}}}$ with $\cX_{_{{\rm sep,inj}}}$, so we obtain $\Mdst(\cX_{_{{\rm sep,inj}}})\cong\cX$-$\PEqu_{_{{\rm sep}}}\cong\cX$-$\Equ_{_{{\rm sep}}}$. Properties from items (1) to (5) remain valid, so they also hold for $\cX$-$\Equ_{_{{\rm sep}}}$.   

\section{Equilogical objects and exact completion}\label{sec3}

The category $\Equ$ of equilogical spaces can also be obtained as a full reflective subcategory of the exact completion \cite{MR1659606,MR1358759,MR678508} $\Top_{_{{\rm ex}}}$ of the category of topological spaces \cite{MR1721095}, and this is a particular instance of a general process to obtain such categories \cite{MR1659606}.

We can describe the exact completion $\cX_{_{{\rm ex}}}$ of $\cX$ as: objects are
{\it pseudo-equivalence relations} on $\cX$, that is, parallel pairs of morphisms $\xymatrix{X_{_{1}} \ar@<0.7ex>[r]^{r_{_{1}}} \ar@<-0.7ex>[r]_{r_{_{2}}} & X_{_{0}}}$ of $\cX$ satisfying \\
(i) {\it reflexivity}: there exists a morphism $r\colon X_{_{0}}\to X_{_{1}}$ such that $r_{_{1}}\cdot r=1_{_{X_{_{0}}}}=r_{_{2}}\cdot r$;
\begin{align*}\xymatrix{ & X_{_{1}} \ar[dl]_{r_{_{1}}} \ar[dr]^{r_{_{2}}} & \\ X_{_{0}} & X_{_{0}} \ar[l]^{1_{_{X_{0}}}} \ar[r]_{1_{_{X_{0}}}} \ar@{-->}[u]_(0.4){r} & X_{_{0}}}\end{align*}
(ii) {\it symmetry}: there exists a morphism $s\colon X_{_{1}}\to X_{_{1}}$ such that $r_{_{1}}\cdot s=r_{_{2}}$ and $r_{_{2}}\cdot s=r_{_{1}}$;  
\begin{align*}\xymatrix{ & X_{_{1}} \ar[dl]_{r_{_{1}}} \ar[dr]^{r_{_{2}}} & \\ X_{_{0}} & X_{_{1}} \ar[l]^{r_{_{2}}} \ar[r]_{r_{_{1}}} \ar@{-->}[u]_(0.4){s} & X_{_{0}}}\end{align*}
(iii) {\it transitivity}: for $r_{_{3}},r_{_{4}}\colon X_{_{2}}\to X_{_{1}}$ a pullback of $r_{_{1}}$,$r_{_{2}}$, there exists a morphism $t\colon X_{_{2}}\to X_{_{1}}$ commuting the following diagram  
\begin{align*}\xymatrix{ & & X_{_{1}} \ar@/_/[dddll]_{r_{_{1}}} \ar@/^/[dddrr]^{r_{_{2}}} & & \\ & & X_{_{2}} \ar[dl]_{r_{_{3}}} \ar[dr]^{r_{_{4}}} \ar@{-->}[u]_(0.45){t} \ar@<-1.5ex>@{}[dd]^(0.25)[@!-45]{\text{\pigpenfont J}} & & \\ & X_{_{1}} \ar[dl]^{r_{_{1}}} \ar[dr]_{r_{_{2}}} & & X_{_{1}} \ar[dl]^{r_{_{1}}} \ar[dr]_{r_{_{2}}} & \\ X_{_{0}} & & X_{_{0}} & & X_{_{0}}.}\end{align*}
A morphism from $\xymatrix{X_{_{1}} \ar@<0.7ex>[r]^{r_{_{1}}} \ar@<-0.7ex>[r]_{r_{_{2}}} & X_{_{0}}}$ to $\xymatrix{Y_{_{1}} \ar@<0.7ex>[r]^{s_{_{1}}} \ar@<-0.7ex>[r]_{s_{_{2}}} & Y_{_{0}}}$ is an equivalence classe $[f]$ of an $\cX$-morphism $f\colon X_{_{0}}\to Y_{_{0}}$ such that there exists $g\colon X_{_{1}}\to Y_{_{1}}$ in $\cX$ satisfying $f\cdot r_{_{i}}=s_{_{i}}\cdot g$, $i=1,2$.
\begin{align*}\xymatrix{X_{_{1}} \ar@<0.7ex>[d]^{r_{_{2}}} \ar@<-0.7ex>[d]_{r_{_{1}}} \ar@{-->}[r]^{g} & Y_{_{1}} \ar@<0.7ex>[d]^{s_{_{2}}} \ar@<-0.7ex>[d]_{s_{_{1}}} \\ X_{_{0}} \ar[r]_{f} & Y_{_{0}}.}\end{align*}
Here two morphisms $f_{_{1}},f_{_{2}}\colon X_{_{0}}\to Y_{_{0}}$ are related if, and only if, there exists a morphism $h\colon X_{_{0}}\to Y_{_{1}}$ such that $f_{_{i}}=s_{_{i}}\cdot h$, $i=1,2$.
\begin{align*}\xymatrix{X_{_{1}} \ar@<0.7ex>[dd]^{r_{_{2}}} \ar@<-0.7ex>[dd]_{r_{_{1}}} & & Y_{_{1}} \ar@<0.7ex>[dd]^{s_{_{2}}} \ar@<-0.7ex>[dd]_{s_{_{1}}} \\ & & \\ X_{_{0}} \ar@<0.7ex>[rr]^{f_{_{1}}} \ar@<-0.7ex>[rr]_{f_{_{2}}} \ar@{-->}[uurr]^{h} & & Y_{_{0}}}\end{align*}

Since it is topological over $\Set$, $\cX$ has a stable factorization system for morphisms given by $({\rm Epi},{\rm RegMono})$ \cite[Remark 15.2(3), Proposition 21.14]{MR1051419} (see also \cite{CHRCCECT}). Let us denote by ${\rm PER}(\cX,{\rm RegMono})$ the full subcategory of $\cX_{_{{\rm ex}}}$ of the pseudo-equivalence relations $\xymatrix{X_{_{1}} \ar@<0.7ex>[r]^{r_{_{1}}} \ar@<-0.7ex>[r]_{r_{_{2}}} & X_{_{0}}}$ such that $\left\langle r_{_{1}},r_{_{2}}\right\rangle\colon X_{_{1}}\to X_{_{0}}\times X_{_{0}}$ is a regular monomorphism.

\begin{lemma}\label{lem30}$\cX$-$\Equ$ and ${\rm PER}(\cX,{\rm RegMono})$ are equivalent.\end{lemma} 

\begin{prf}{\rm For each equilogical object $\left\langle X,\equiv_{_{|X|}}\right\rangle$, consider $E_{_{X}}=\{(x,x')\in|X|\times|X| \ | \ x\equiv_{_{|X|}}x'\}$ and the source $(\pi_{_{i}}^{X}\colon E_{_{X}}\to|X|)_{_{i=1,2}}$ of the projections from $E_{_{X}}$ onto $|X|$. Take its $|$-$|$-initial lifting, which by abuse of notation we denote by $(\pi_{_{i}}^{X}\colon E_{_{X}}\to X)_{_{i=1,2}}$. Hence $\xymatrix{E_{_{X}} \ar@<0.7ex>[r]^{\pi_{_{1}}^{X}} \ar@<-0.7ex>[r]_{\pi_{_{2}}^{X}} & X}$ belongs to ${\rm PER}(\cX,{\rm RegMono})$ and each morphism $[f]\colon\X\to\Y$ in $\cX$-$\Equ$ is a valid morphism 
\begin{align*}[f]\colon(E_{_{X}},X,\pi_{_{1}}^{X},\pi_{_{2}}^{X})\to(E_{_{Y}},Y,\pi_{_{1}}^{Y},\pi_{_{2}}^{Y})\end{align*} 
in ${\rm PER}(\cX,{\rm RegMono})$.

That correspondence defines a functor which is fully faithful and, for a pseudo-equivalence relation $\xymatrix{X_{_{1}} \ar@<0.7ex>[r]^{r_{_{1}}} \ar@<-0.7ex>[r]_{r_{_{2}}} & X_{_{0}}}$ in ${\rm PER}(\cX,{\rm RegMono})$, define the equilogical object $\left\langle X_{_{0}},\equiv_{_{|X_{_{0}}|}}\right\rangle$ by 
\begin{align*}x_{_{0}}\equiv_{_{|X_{_{0}}|}}x_{_{0}}' \ \Longleftrightarrow \  (\exists \ \text{(unique)} \ x_{_{1}}\in X_{_{1}}) \ |r_{_{1}}|(x_{_{1}})=x_{_{0}} \ \with \ |r_{_{2}}|(x_{_{1}})=x_{_{0}}',\end{align*} 
for each $x_{_{0}},x_{_{0}}'\in X_{_{0}}$. Then $\xymatrix{E_{_{X_{_{0}}}} \ar@<0.7ex>[r]^{\pi_{_{1}}^{X_{_{0}}}} \ar@<-0.7ex>[r]_{\pi_{_{2}}^{X_{_{0}}}} & X_{_{0}}}$ is isomorphic to $\xymatrix{X_{_{1}} \ar@<0.7ex>[r]^{r_{_{1}}} \ar@<-0.7ex>[r]_{r_{_{2}}} & X_{_{0}}}$ in ${\rm PER}(\cX,{\rm RegMono})$ and the functor is essentially surjective.}\end{prf}

Hence \cite[Theorem 4.3]{MR1659606} states that 
\begin{theorem}\label{teo31}$\cX$-$\Equ\cong{\rm PER}(\cX,{\rm RegMono})$ is a full reflective subcategory of $\cX_{_{{\rm ex}}}$; the reflector preserves finite products and commutes with change of base in the codomain.\end{theorem} 

Next we wish to prove that $\cX$-$\Equ$ is cartesian closed, so by Theorem \ref{teo31} and \cite[Theorem 1.2]{MR785032}, it suffices to show that $\cX_{_{{\rm ex}}}$ is cartesian closed. To do so, we will apply the following result derived from \cite[Theorem 1, Lemma 4]{MR1721095} (see also \cite[Theorem 1.1]{CHRCCECT}).

\begin{theorem}\label{teo1}Let $\cX$ be a complete, infinitely extensive and well-powered category with factorizations $({\rm RegEpi},{\rm Mono})$ such that $f\times1$ is an epimorphism whenever $f$ is a regular epimorphism. Then $\cX_{_{{\rm ex}}}$ is cartesian closed provided $\cX$ is weakly cartesian closed.\end{theorem}

Since $\cX$ is topological over $\Set$, in order to use the latter theorem, we will assume that \\[0.4cm]
{\bf (f)} {\it $\cX$ is infinitely extensive;} \\[0.4cm]
more, assuming also conditions {\bf (a)} to {\bf (e)} from the previous section, following the same steps of the proofs of \cite[Theorems 5.3 and 5.5]{CHRCCECT}, we deduce the following result.

\begin{proposition}\label{prop302}$\cX$ is weakly cartesian closed.\end{proposition} 

Furthermore, we can verify that $\cX_{_{{\rm ex}}}$ is actually locally cartesian closed. Consider the restriction functor $|$-$|\colon\cX_{_{{\rm inj}}}\to\Pfn$, where $\Pfn$ is the category of sets and partial functions. The category $\mathcal{F}(\cX_{_{{\rm inj}}},|$-$|)$, or simply $\mathcal{F}(\cX_{_{{\rm inj}}})$, is described in \cite{MR1787592} as follows: objects are triples $(X,A,\sigma\colon A\to|X|)$, where $X$ is an injective object of $\cX$, $A$ is a set and $\sigma$ is a function; a morphism $f\colon(X,A,\sigma\colon A\to|X|)\to(Y,B,\delta\colon B\to|Y|)$ is a map $f\colon A\to B$ such that there exists $g\colon X\to Y$ in $\cX$ commuting the diagram
\begin{align*}\xymatrix{A \ar[d]_{\sigma} \ar[r]^{f} & B \ar[d]^{\delta} \\ |X| \ar[r]_{|g|} & |Y|.}\end{align*} 

\begin{proposition}\label{prop31}The categories $\cX$ and $\mathcal{F}(\cX_{_{{\rm inj}}})$ are equivalent.\end{proposition}

\begin{prf}{\rm Define the functor $G\colon\cX\to\mathcal{F}(\cX_{_{{\rm inj}}})$ by 
\begin{align*}GX=(\hat{X},|X|,\sigma_{_{X}}=|\yoneda_{_{X}}|\colon|X|\to|\hat{X}|),\end{align*} 
where $\yoneda_{_{X}}$ is the $|$-$|$-initial morphism assured by condition {\bf (c)} in the previous section; for each morphism $f\colon X\to Y$, injectivity of $\hat{Y}$ implies the existence of a morphism $g\colon\hat{X}\to\hat{Y}$ extending $\yoneda_{_{Y}}\cdot f$ along $\yoneda_{_{X}}$
\begin{align*}\xymatrix{X \ar[d]_{\yoneda_{_{X}}} \ar[r]^{f} & Y \ar[d]^{\yoneda_{_{Y}}} \\ \hat{X} \ar[r]_{g} & \hat{Y},}\end{align*}
hence we set $Gf=|f|$. $G$ is faithful and to see it is full, let $f\colon|X|\to|Y|$ be a map commuting the diagram
\begin{align*}\xymatrix{|X| \ar[d]_{|\yoneda_{_{X}}|} \ar[r]^{f} & |Y| \ar[d]^{|\yoneda_{_{Y}}|} \\ |\hat{X}| \ar[r]_{|g|} & |\hat{Y}|,}\end{align*} 
for some $g\colon\hat{X}\to\hat{Y}$ in $\cX$, then $|$-$|$-initiality of $\yoneda_{_{Y}}$ implies the existence of a unique $\overline{f}\colon X\to Y$ such that $G\overline{f}=|\overline{f}|=f$. 

For essential surjectivity, let $(X,A,\sigma\colon A\to|X|)$ in $\mathcal{F}(\cX_{_{{\rm inj}}})$ and take the $|$-$|$-initial lifting of $\sigma$, that we denote by $\sigma_{_{{\rm ini}}}\colon A_{_{{\rm ini}}}\to X$, so $|A_{_{{\rm ini}}}|=A$ and $|\sigma_{_{{\rm ini}}}|=\sigma$. Hence $G A_{_{{\rm ini}}}=(\hat{A}_{_{{\rm ini}}},A,|\yoneda_{_{A_{_{{\rm ini}}}}}|\colon A\to|\hat{A}_{_{{\rm ini}}}|)$ and we verify that the identity map $1_{_{A}}\colon A\to A$ is a morphism from $(X,A,\sigma)$ to $(\hat{A}_{_{{\rm ini}}},A,|\yoneda_{_{A_{_{{\rm ini}}}}}|)$, and vice-versa. The latter fact comes readly from injectivity of $\hat{A}_{_{{\rm ini}}}$ and $|$-$|$-initiality of $\sigma_{_{{\rm ini}}}$: 
\begin{align*}\xymatrix{A \ar[d]_{\sigma} \ar[r]^{1_{_{A}}} & A \ar[d]^{|\yoneda_{_{A_{_{{\rm ini}}}}}|} \\ |X| \ar[r]_{|g_{_{1}}|} & |\hat{A}_{_{{\rm ini}}}|,}\end{align*}
for some morphism $g_{_{1}}\colon X\to\hat{A}_{_{{\rm ini}}}$, and by injectivity of $X$ and $|$-$|$-initiality of $\yoneda_{_{A_{_{{\rm ini}}}}}$:
\begin{align*}\xymatrix{A \ar[d]_{|\yoneda_{_{A_{_{{\rm ini}}}}}|} \ar[r]^{1_{_{A}}} & A \ar[d]^{\sigma} \\ |\hat{A}_{_{{\rm ini}}}| \ar[r]_{|g_{_{2}}|} & |X|,}\end{align*}
for some morphism $g_{_{2}}\colon\hat{A}_{_{{\rm ini}}}\to X$. Therefore, $G A_{_{{\rm ini}}}\cong(X,A,\sigma\colon A\to|X|)$ in $\mathcal{F}(\cX_{_{{\rm inj}}})$.}\end{prf}

Since $\cX_{_{{\rm inj}}}$ is (weakly) cartesian closed, as shown in \cite{MR1787592}, $\cX\cong\mathcal{F}(\cX_{_{{\rm inj}}})$ has all weak simple products (in particular it is weakly cartesian closed), and more, $\cX\cong\mathcal{F}(\cX_{_{{\rm inj}}})$ is weakly locally cartesian closed, i.e. it has weak dependent products, whence by \cite[Theorem 3.3]{MR1787592}, $\cX_{_{{\rm ex}}}\cong\mathcal{F}(\cX_{_{{\rm inj}}})_{_{{\rm ex}}}$ is locally cartesian closed.

Therefore, by Theorem \ref{teo31}, we conclude that $\cX$-$\Equ$ is locally cartesian closed (see for instance \cite[III-Corollary 4.6.2]{MR3307673}), and, being complete and cocomplete, one may ask whether this category is actually a quasitopos.

As discussed in \cite{esfsgr}, {\it ``... the full subcategory of $\C_{_{{\rm ex}}}$ consisting of those equivalence spans which are kernel pairs in $\C$ gives the free regular completion $\C_{_{{\rm reg}}}$ of $\C$.''}, where in that context {\it equivalence span} means pseudo-equivalence relation. Hence, similar to what is observed in \cite{ecatmm} for topological spaces, the category $\cX$-$\Equ$, presented by ${\rm PER}(\cX,{\rm RegMono})$, is equivalent to the regular completion $\cX_{_{{\rm reg}}}$ of $\cX$.

It is easy to depict the latter equivalence using the classical description of $\cX_{_{{\rm reg}}}$ \cite{MR1358759}: objects are $\cX$-morphisms $f\colon X\to Y$, and a morphism from $f\colon X\to Y$ to $g\colon Z\to W$ is an equivalence class $[l]$ of an $\cX$-morphism $l\colon X\to Z$ such that $g\cdot l\cdot f_{_{0}}=g\cdot l\cdot f_{_{1}}$, where $f_{_{0}},f_{_{1}}$ form the kernel pair of $f$. 
\begin{align*}\xymatrix{{\rm Ker}(f) \ar[d]_{f_{_{0}}} \ar[r]^(0.55){f_{_{1}}} \ar@{}[dr]|(0.3){\text{\pigpenfont J}} & X \ar[d]^{f} \\ X \ar[r]_{f} & Y}\end{align*}
Two such arrows $l$ and $m$ are equivalent if $g\cdot l=g\cdot m$.
\begin{align*}\xymatrix{X \ar[dd]_{f} & & Z \ar[dd]^{g} \\ \ar[rr]^{[l]} & & \\ Y & & W}\end{align*}

\begin{lemma}\label{lem32}$\cX_{_{{\rm reg}}}$ and ${\rm PER}(\cX,{\rm RegMono})$ are equivalent.\end{lemma}

\begin{prf}{\rm Define $F\colon\cX_{_{{\rm reg}}}\to{\rm PER}(\cX,{\rm RegMono})$ as in the diagram below,
\begin{align*}\xymatrix{(f\colon X\to Y) \ar@{|->}[rr]^{} \ar[d]_{[l]} & & ({\rm Ker}(f),X,f_{_{0}},f_{_{1}}) \ar[d]^{[l]} \\ (g\colon Z\to W) \ar@{|->}[rr]_{} & & ({\rm Ker}(g),Z,g_{_{0}},g_{_{1}})}\end{align*} 
so it is a well-defined functor, since $l\colon X\to Z$ satisfies $g\cdot l\cdot f_{_{0}}=g\cdot l\cdot f_{_{1}}$ if, and only if, there exists a unique $\overline{l}\colon{\rm Ker}(f)\to{\rm Ker}(g)$ such that $g_{_{0}}\cdot\overline{l}=l\cdot f_{_{0}}$ and $g_{_{1}}\cdot\overline{l}=l\cdot f_{_{1}}$.
\begin{align*}\xymatrix{{\rm Ker}(f) \ar@/_1.5pc/[ddr]_{l\cdot f_{_{0}}} \ar[dr]^{\overline{l}} \ar@/^1.5pc/[drr]^{l\cdot f_{_{1}}} & & \\ & {\rm Ker}(g) \ar[d]_{g_{_{0}}} \ar[r]^{g_{_{1}}} \ar@{}[dr]|(0.3){\text{\pigpenfont J}} & Z \ar[d]^{g} \\ & Z \ar[r]_{g} & W}\end{align*} 
Then $F$ is fully faithful, and it is essentially surjective because each pseudo-equivalence relation $\xymatrix{X_{_{1}} \ar@<0.7ex>[r]^{r_{_{1}}} \ar@<-0.7ex>[r]_{r_{_{2}}} & X_{_{0}}}$ with $\left\langle r_{_{1}},r_{_{2}}\right\rangle\colon X_{_{1}}\to X_{_{0}}\times X_{_{0}}$ a regular monomorphism is seen to form the kernel pair of the $|$-$|$-final lifting $\overline{p}\colon X_{_{0}}\to\tilde{X}_{_{0}}$ of the projection map $p\colon|X_{_{0}}|\to|X_{_{0}}|/\hspace{-0.1cm}\sim$, where $\sim$ is the equivalence relation on $|X_{_{0}}|$ defined in the proof of Lemma \ref{lem30}.}\end{prf}  

We now intend to apply \cite[Corollary 8.4.2]{ecatmm}; by condition {\bf (f)} and Proposition \ref{prop31}, $\cX$ is an (infinitely) extensive weakly locally cartesian closed category, hence we are only missing the {\it chaotic situation} described right after \cite[Lemma 7.3.3]{ecatmm}. This comes from the observation that the topos $\Set$ is a mono-localization of $\cX$, since the topological functor $|$-$|\colon\cX\to\Set$ is faithful, preserves finite limits and has a full embedding as a right adjoint \cite[Proposition 21.12]{MR1051419}. Therefore, by Lemma \ref{lem30}, Lemma \ref{lem32} and \cite[Corollary 8.4.2]{ecatmm}, we conclude: 

\begin{theorem}\label{teo2}$\cX$-$\Equ$ is a quasitopos.\end{theorem} 

\section{The case \texorpdfstring{$\cX$}{X}\texorpdfstring{$=$}{=}\texorpdfstring{$\Cat{(\T,\V)}$}{(T,V)-Cat}}\label{sec4}

We briefly introduce the $\Cat{(\T,\V)}$ setting, and refer the reader to the reference \cite{MR1957813} for details (see also \cite{MR3307673}). 

Although introduced in a more general setting, we are interested here in the case when \\
$\bullet$ $\V=(\V,\otimes,k)$ is a commutative unital quantale (see for instance \cite[II-Section 1.10]{MR3307673}) which is also a {\it Heyting algebra} (so that the operation infimum $\wedge$ also has a right adjoint), and \\
$\bullet$ $\T=(T,m,e)\colon\Set\to\Set$ is a monad satisfying the Beck-Chevalley condition ($T$ preserves weak pullbacks and the naturality squares of $m$ are weak pullbacks \cite{MR3175322}) that is laxly extended to the pre-ordered category $\Rell{\V}$, which has as objects sets and as morphisms $\V$-relations $r\colon X\rel Y$, i.e. $\V$-valued maps $r\colon X\times Y\to\V$.

Hence we assume that there exists a functor $T\colon\Rell{\V}\to\Rell{\V}$ extending $T$, by abuse of notation denoted by the same letter, that commutes with involution: $T(r^{\circ})=(Tr)^{\circ}$, for each $r\colon X\rel Y\in\Rell{\V}$, where $r^{\circ}(y,x)=r(x,y)$, for each $(x,y)\in X\times Y$.

The functor $T$ turns $m$ and $e$ into oplax transformations, meaning that the naturality diagrams become:
\begin{align*}\xymatrix{X \ar[r]^{e_{X}} \ar[d]|(0.45){\object@{|}}_(0.45){r} \ar@{}[dr]|{\le} & TX \ar[d]|(0.45){\object@{|}}^(0.45){Tr} & T^{2}X \ar[l]_{m_{X}} \ar[d]|(0.45){\object@{|}}^(0.45){T^{2}r} \ar@{}[dl]|{\ge} \\ Y \ar[r]_{e_{Y}} & TY & T^{2}Y, \ar[l]^{m_{Y}}}\end{align*}
for each $\V$-relation $r\colon X\rel Y$. 

Hence we have a lax monad on $\Rell{\V}$ \cite{MR2116322} and $\Cat{(\T,\V)}$ is defined as the category of Eilenberg-Moore lax algebras for that lax monad: objects are pairs $(X,a)$, where $X$ is a set and $a\colon TX\rel X$ is a $\V$-relation, which is reflexive and transitive.
\begin{align*}\xymatrix{X \ar[r]^{e_{X}} \ar@/_1pc/[dr]_{1_{X}} & TX \ar[d]|(0.45){\object@{|}}^(0.45){a} & T^{2}X \ar[l]|{\object@{|}}_{Ta} \ar[d]^{m_{X}} \ar@{}[dl]|{\le} \\ \ar@{}[ur]|{\le} & X & TX \ar[l]|(0.45){\object@{|}}^(0.45){a}}\end{align*}
Such pairs are called {\it $(\T,\V)$-categories}; a morphism from $(X,a)$ to $(Y,b)$ is a map $f\colon X\to Y$ commuting the diagram below.
\begin{align*}\xymatrix{TX \ar[d]|(0.45){\object@{|}}_(0.45){a} \ar[r]^{Tf} \ar@{}[dr]|{\le} & TY \ar[d]|(0.45){\object@{|}}^(0.45){b} \\ X \ar[r]_{f} & Y}\end{align*}
Such a map is called {\it $(\T,\V)$-functor}. 

We are also going to restrict ourselves to the case that the extension $T$ to $\Rell{\V}$ is determined by a $\T$-algebra structure map $\xi\colon TV\to V$, so we are in the setting of {\it topological theories} \cite{MR2355608} (see also \cite{MR3330902}), hence $\V$ has a $(\T,\V)$-category structure given by the composite
\begin{align*}\xymatrix{T\V \ar[rr]^{\xi} & & \V \ar[rr]|(0.45){\object@{|}}^(0.45){{\rm hom}} & & \V,}\end{align*}
where ${\rm hom}\colon\V\times\V\to\V$ is the left adjoint of $\otimes$, so 
\begin{align*}u\otimes v\leq w \ \Longleftrightarrow \ u\leq{\rm hom}(v,w),\end{align*}
for each $u,v,w$ in the quantale $\V$.

The forgetful functor $|$-$|\colon\Cat{(\T,\V)}\to\Set$ is topological \cite{MR1990036,MR1957813}, and before we provide examples of categories given by $\Cat{(\T,\V)}$, we verify that, for suitable monad $\T$ and quantale $\V$ satisfying the conditions assumed so far in this section, $\Cat{(\T,\V)}$ satisfies all conditions {\bf (a)} to {\bf (f)} from the two previous sections. In each item, we highlight the properties that are needed in order to achieve the respective condition, adding the references where that is proved. \vspace{0.4cm}

\noindent{\bf (a)} {\it $\Cat{(\T,\V)}$ is pre-ordered enriched.} For $(\T,\V)$-categories $(X,a)$ and $(Y,b)$, consider the following relation on the set of $(\T,\V)$-functors  from $(X,a)$ to $(Y,b)$:
\begin{align*}f\leq g \ \Longleftrightarrow \ \forall x\in X, \ k\leq b(e_{_{Y}}(f(x)),g(x)).\end{align*}
This determines a pre-order, first defined in \cite{MR1957813}, which is compatible with composition of $(\T,\V)$-functors. One can also check that a $(\T,\V)$-category $(X,a)$ is separated if, and only if, the following pre-order on $X$ is anti-symmetric:
\begin{align*}x\leq x' \ \Longleftrightarrow \ k\leq a(e_{_{X}}(x),x')\end{align*}
(see \cite[III-Proposition 3.3.1]{MR3307673}). \vspace{0.4cm}

\noindent{\bf (b)} {\it $|$-$|$-initial $(\T,\V)$-functors reflect the order.} Let $(X,a),(Y,b)$ be $(\T,\V)$-categories, $x,x'\colon\one\to(X,a)$ be $(\T,\V)$-functors, where $\one=(\{*\},e_{_{\{*\}}}^{\circ})$ (the discrete structure on the singleton \cite[III-Section 3.2]{MR3307673}), and $f\colon(X,a)\to(Y,b)$ an $|$-$|$-initial $(\T,\V)$-functor such that $f\cdot x\simeq f\cdot x'$; $|$-$|$-initiality of $f$ means that $a(\x,x)=b(Tf(\x),f(x))$, for each $\x\in TX$, $x\in X$. We calculate:
\begin{align*}\begin{array}{rcll} k & \leq & b(e_{_{Y}}(f\cdot x(*)),f\cdot x'(*)) & (\text{definition of $f\cdot x\leq f\cdot x'$}) \\
& \leq & b(Tf\cdot e_{_{X}}(x(*)),f\cdot x'(*)) & (\text{composition is associative, $e$ is natural}) \\
& \leq & a(e_{_{X}}(x(*)),x'(*)) & (\text{$f$ is $|$-$|$-initial}),\end{array}\end{align*}
so $x\leq x'$ and in the same fashion $x'\leq x$, thus $x\simeq x'$. \vspace{0.4cm}

\noindent{\bf (c)} {\it $\Cat{(\T,\V)}$ has enough injectives.} The tensor product $\otimes$ of $\V$ induces a functor $\otimes\colon\Cat{(\T,\V)}\times\Cat{(\T,\V)}\to\Cat{(\T,\V)}$, with 
\begin{align*}(X,a)\otimes(Y,b)=(X\times Y,c),\end{align*}
where, for each $\w\in T(X\times Y)$, $(x,y)\in X\times Y$,
\begin{align*}c(\w,(x,y))=a(T\pi_{_{X}}(\w),x)\otimes b(T\pi_{_{Y}}(\w),y),\end{align*}
and $\pi_{_{X}},\pi_{_{Y}}$ are the projections from $X\times Y$ onto $X$ and $Y$, respectively. The following facts can be found in \cite{MR2491799, MR2729224, MR3416170}: for each $(\T,\V)$-category $(X,a)$, $a\colon TX\rel X$ defines a $(\T,\V)$-functor
\begin{align*}a\colon X^{{\rm op}}\otimes X\to\V,\end{align*}
where $X^{{\rm op}}=(TX,m_{_{X}}\cdot(Ta)^{\circ}\cdot m_{_{X}})$; the $\otimes$-exponential mate $\yoneda_{_{X}}\colon X\to\V^{X^{{\rm op}}}$ of $a$ is fully faithful; the $(\T,\V)$-category $PX=\V^{X^{{\rm op}}}$ is injective and if $(X,a)$ is separated, so is $PX$. \vspace{0.4cm}

\noindent{\bf (d)} {\it Injectives are exponentiable.} Conditions under which injectivity implies exponentiability in $\Cat{(\T,\V)}$ are studied in \cite{CHRCCECT}. We recall them next. Consider the maps
\begin{align}\label{inj1}\xymatrix{\V\otimes\V \ar[rr]^{\otimes} & & \V} \ \ \ \text{and} \ \ \ \xymatrix{X \ar[rr]^{(-,u)} & & X\otimes\V},\end{align}
with $(\V,\hom_{_{\xi}}),(X,a)\in\Cat{(\T,\V)}$. Define also for a $\V$-relation $r\colon X\rel Y$ and $u\in\V$, the $\V$-relation $r\otimes u\colon X\rel Y$ given by
\begin{align}\label{inj2}(r\otimes u)(x,y)=r(x,y)\otimes u,\end{align}
for each $(x,y)\in X\times Y$. As a final condition, assume that, for all $u,v,w\in\V$,
\begin{align}\label{inj3}w\wedge(u\otimes v)=\{u'\otimes v' \ | \ u'\leq u, \ v'\leq v, \ u'\otimes v'\leq w\},\end{align}
which is equivalent to exponentiability of injective $\V$-categories (see \cite[Theorem 5.3]{MR2981702}). Then \cite[Theorem 5.4]{CHRCCECT} says the following:
\begin{theorem}Suppose that: the maps $\otimes$ and $(-,u)$ in (\ref{inj1}) are $(\T,\V)$-functors; for every injective $(\T,\V)$-category $(X,a)$ and every $u\in\V$, $T(a\otimes u)=Ta\otimes u$, with those $\V$-relations defined as in (\ref{inj2}); and (\ref{inj3}) holds. Then every injective $(\T,\V)$-category is exponentiable in $\Cat{(\T,\V)}$.\end{theorem} 
{\bf (e)} {\it The reflector from $\Cat{(\T,\V)}$ to $\Cat{(\T,\V)}_{_{{\rm sep}}}$ preserves finite products.} This is proved in \cite[Proposition 5.4]{CHRCCECT}. \vspace{0.4cm}

\noindent{\bf (f)} {\it $\Cat{(\T,\V)}$ is infinitely extensive.} This is proved in \cite{MR2248546} under the condition that $T$ is a taut functor \cite{MR1902090}, what comes for free from the assumption that $T$ preserves weak pullbacks. \vspace{0.4cm}

To give examples of categories satisfying all the conditions above, we consider: \vspace{0.4cm}

\noindent$\bullet$ the identity monad $\mathbb{I}=({\rm Id},\one,\one)$ on $\Set$ laxly extended to the identity lax monad on $\Rell{\V}$; \vspace{0.4cm}

\noindent$\bullet$ the ultrafilter monad $\U$ with the Barr extension to $\Rell{\V}$ \cite[IV-Corollary 2.4.5]{MR3307673}, with $\V$ integral and {\it completely distributive} (see, for instance, \cite[II-Section 1.11]{MR3307673}); \vspace{0.4cm}

\noindent$\bullet$ the list monad (or free monoid monad) $\mL=(L,m,e)$ (see \cite[II-Examples 3.1.1(2)]{MR3307673}), with the extension $L\colon\Rell{\V}\to\Rell{\V}$ that sends each $r\colon X\rel Y$ to the $\V$-relation $Lr\colon LX\rel LY$ given by
\[
  Lr((x_{_{1}},\dots,x_{_{n}}),(y_{_{1}},\dots,y_{_{m}}))=
  \begin{cases}
    r(x_{_{1}},y_{_{1}})\otimes\dots\otimes r(x_{_{n}},y_{_{n}}), \ \text{if $n=m$}\\
    \bot, \ \text{if $n\neq m$};
  \end{cases}
\]
$\bullet$ the monad $\mM=(-\times M,m,e)$, for a monoid $(M,\cdot,1_{_{M}})$, with $m_{_{X}}\colon X\times M\times M\to X\times M$ given by $m_{_{X}}(x,a,b)=(x,a\cdot b)$ and $e_{_{X}}\colon X\to X\times M$ given by $e_{_{X}}(x)=(x,1_{_{M}})$ (see \cite[V-Section 1.4]{MR3307673}). The extension $-\times M:\Rell{\V}\to\Rell{\V}$ sends the $\V$-relation $r\colon X\rel Y$ to the $\V$-relation $r\times M\colon X\times M\rel Y\times M$ with 
\[
  r\times M((x,a),(y,b))=
  \begin{cases}
    r(x,y), \ \text{if $a=b$,}\\
    \bot, \ \text{if $a\neq b$.}
  \end{cases}
\]

\noindent As well as the quantales: $\two=(\{\bot,\top\},\wedge,\top)$, $\categ{P}_{_{+}}=([0,\infty]^{\rm{op}},+,0)$, $\categ{P}_{_{{\rm max}}}=([0,\infty]^{\rm{op}},{\rm max},0)$, $\two^{2}=(\{\bot, u, v, \top\},\wedge,\top)$ (the {\it diamond lattice} \cite[II-Exercise 1.H]{MR3307673}) and $\Delta$ (the quantale of distribution functions \cite{MR2981702}). We assemble the table: 
\begin{align}\label{tab1}\text{\begin{tabular}{|l|c|c|c|c|} 
\hline
\backslashbox{$\V$}{$\T$} & $\I$       & $\U$              & $\mL$       & $\mM$                \\ \hline
$\two$                    & $\Ord$     & $\Top$            & $\MultiOrd$ & $\Cat{(\mM,\two)}$   \\ \hline
$\categ{P}_{_{+}}$        & $\Met$     & $\App$            &             &                      \\ \hline
$\categ{P}_{_{{\rm max}}}$& $\UltMet$  & $\NA\text{-}\App$ &             &                      \\ \hline
$\two^{2}$                & $\BiRel$   & $\BiTop$          &             &                      \\ \hline
$\Delta$                  & $\ProbMet$ &                   &             &                      \\ \hline\end{tabular}}\end{align}

\noindent$\bullet$ $\Ord$ is the category of pre-ordered spaces, \vspace{0.1cm}

\noindent$\bullet$ $\Met$ is the category of Lawvere generalized metric spaces \cite{MR1925933}, \vspace{0.1cm}

\noindent$\bullet$ $\UltMet$ is the full subcategory of $\Met$ of ultra-metric spaces \cite[III-Exercise 2.B]{MR3307673}, \vspace{0.1cm}

\noindent$\bullet$ $\BiRel$ is the one of sets and birelations \cite[III-Examples 1.1.1(3)]{MR3307673}, \vspace{0.1cm}

\noindent$\bullet$ $\ProbMet$ is the category of probabilistic metric spaces \cite{MR2981702}, \vspace{0.1cm}

\noindent$\bullet$ $\Top$ is the usual category of topological spaces and continuous functions, \vspace{0.1cm}

\noindent$\bullet$ $\App$ is that of Lowen's approach spaces \cite{MR1472024}, and \vspace{0.1cm}

\noindent$\bullet$ $\NA\text{-}\App$ is the full subcategory of $\App$ of non-Archimedean approach spaces studied in details in \cite{MR3731477}, and denoted in \cite{MR3227304} by $\UApp$, \vspace{0.1cm}

\noindent$\bullet$ $\BiTop$ is the category of bitopological spaces and bicontinuous maps \cite[III-Exercise 2.D]{MR3307673}, \vspace{0.1cm}

\noindent$\bullet$ $\MultiOrd$ is the category of multi-ordered sets \cite[V-Section 1.4]{MR3307673}, and \vspace{0.1cm} 

\noindent$\bullet$ $\Cat{(\mM,\two)}$ can be interpreted as the category of $M$-labelled ordered sets \cite[V-Section 1.4]{MR3307673}.

For instance, an object of $\Ord$-$\Equ$ is a pre-ordered set $(X,\leq)$ together with an equivalence relation $\equiv_{_{X}}$ on $X$; separatedness of $(X,\leq)$ means that $\leq$ is anti-symmetric, so the objects of $\Ord$-$\Equ_{_{{\rm sep}}}$ are partially ordered sets equipped with equivalence relations on their underlying sets. Further, a partial equilogical separated object in $\Ord$-$\PEqu_{_{{\rm sep}}}$ is a complete lattice (injective ordered set) together with an equivalence relation on the underlying set. In the same fashion, the objects of the category $\Mdst(\Ord_{_{{\rm sep,inj}}})$ are triples $(A,(X,\leq),E_{_{A}})$, with $A$ a set, $E_{_{A}}\colon A\to\mathcal{P}X$ a function, and $(X,\leq)$ a complete lattice.

Furthermore, from Section \ref{sec3} we conclude that, together with $\Top$, all the other categories in Table \ref{tab1} are weakly locally cartesian closed and their exact completions are locally cartesian closed categories; moreover, their categories of equilogical objects, which are equivalent to their regular completions, are quasitoposes that fully embed the original categories.   

Concerning four of those categories, we also have adjunctions
\begin{align*}\xymatrix{\Top \ar@{^{(}->}@<-1.1ex>[rr]_{} \ar@<1.1ex>[dd]^{} & & \App \ar@{.>}@<-1.1ex>[ll]^[right]{\dashv} \ar@<1.1ex>[dd]_{\dashv} \\ & & \\ \Ord \ar@{_{(}.>}@<1.1ex>[uu]_{\dashv} \ar@{^{(}->}@<-1.1ex>[rr]^[right]{\dashv} & & \Met, \ar@{.>}@<-1.1ex>[ll]^{} \ar@{_{(}.>}@<1.1ex>[uu]^{}}\end{align*}   
where both solid and dotted diagrams commute, the hook-arrows are full embeddings and the two full embeddings $\Ord\hookrightarrow\App$ coincide (see \cite[III-Section 3.6]{MR3307673}). One can see that those adjunctions extend to the respective categories of equilogical objects,
\begin{align*}\xymatrix{\Equ \ar@{^{(}->}@<-1.1ex>[rr]_{} \ar@<1.1ex>[dd]^{} & & \App\text{-}\Equ \ar@{.>}@<-1.1ex>[ll]^[right]{\dashv} \ar@<1.1ex>[dd]_{\dashv} \\ & & \\ \Ord\text{-}\Equ \ar@{_{(}.>}@<1.1ex>[uu]_{\dashv} \ar@{^{(}->}@<-1.1ex>[rr]^[right]{\dashv} & & \Met\text{-}\Equ \ar@{.>}@<-1.1ex>[ll]^{} \ar@{_{(}.>}@<1.1ex>[uu]^{}}\end{align*}
and we describe them now. \vspace{0.4cm}

\noindent(1) {\it $\Ord$-$\Equ$ to $\Met$-$\Equ$.} Each ordered equilogical object $\left\langle (X,\leq),\equiv_{_{X}}\right\rangle$ is taken to $\left\langle (X,d_{_{\leq}}),\equiv_{_{X}}\right\rangle$, where the metric $d_{_{\leq}}$ is given by 
\begin{align*}d_{_{\leq}}(x,x')=\left\{\begin{array}{rl}0, & \text{if $x\leq x'$} \\ \infty, & \text{otherwise,}\end{array}\right.\end{align*}
for each $x,x'\in X$. The left adjoint of this inclusion assigns $\left\langle (X,\leq_{_{d}}),\equiv_{_{X}}\right\rangle$ to $\left\langle (X,d),\equiv_{_{X}}\right\rangle$, with $x\leq_{_{d}}x'$ if and only if $d(x,x')<\infty$, for each $x,x'\in X$. Hence the category $\Ord$-$\Equ$ is fully embedded in $\Met$-$\Equ$ as the metric equilogical objects $\left\langle (X,d),\equiv_{_{X}}\right\rangle$ for which there exists an order $\leq$ on $X$ such that $d=d_{_{\leq}}$. \vspace{0.4cm}

\noindent(2) {\it $\Ord$-$\Equ$ to $\Equ$.} Each $\left\langle (X,\leq),\equiv_{_{X}}\right\rangle$ is taken to $\left\langle (X,\tau_{_{\leq}}),\equiv_{_{X}}\right\rangle$, where $\tau_{_{\leq}}$ is the Alexandroff topology: open sets are generated by the down-sets $\downarrow\hspace{-0.15cm}x$, $x\in X$. For its right adjoint, to an equilogical space $\left\langle (X,\tau),\equiv_{_{X}}\right\rangle$ is assigned $\left\langle (X,\leq_{_{\tau}}),\equiv_{_{X}}\right\rangle$, where $\leq_{_{\tau}}$ is the specialization order of $(X,\tau)$: for each $x,x'\in X$, $x\leq x'$ if and only if $\dot{x}\rightarrow x'$, where $\dot{x}$ denotes the principal ultrafilter on $x$, and $\rightarrow$ denotes the convergence relation between ultrafilters and points of $X$ determined by $\tau$; observe that this is the induced order described in item {\bf (a)} above. Hence the category $\Ord$-$\Equ$ is fully embedded in $\Equ$ as the equilogical spaces $\left\langle (X,\tau),\equiv_{_{X}}\right\rangle$ for which there exists an order $\leq$ on $X$ such that $\tau=\tau_{_{\leq}}$, and those are exactly the Alexandroff spaces: arbitrary intersections of open sets are open (see \cite[II-Example 5.10.5, III-Example 3.4.3(1)]{MR3307673}). \vspace{0.4cm}

\noindent(3) {\it $\Met$-$\Equ$ to $\App$-$\Equ$.} A metric equilogical object $\left\langle (X,d),\equiv_{_{X}}\right\rangle$ becomes an approach equilogical one $\left\langle (X,\delta_{_{d}}),\equiv_{_{X}}\right\rangle$, where the approach distance is given by $\delta_{_{d}}(x',A)={\rm inf}\{d(x,x') \ | \ x\in A\}$, for each $x'\in X$, $A\in\mathcal{P}X$ \cite[III-Examples 2.4.1(1)]{MR3307673}. The right adjoint of this embedding assigns $\left\langle (X,d_{_{\delta}}),\equiv_{_{X}}\right\rangle$ to $\left\langle (X,\delta),\equiv_{_{X}}\right\rangle$, where $d_{_{\delta}}(x,x')={\rm sup}\{\delta(x',A) \ | \ x\in A\in\mathcal{P}X\}$, for each $x,x'\in X$. Hence $\Met$-$\Equ$ is identified within $\App$-$\Equ$ as the approach equilogical objects $\left\langle (X,\delta),\equiv_{_{X}}\right\rangle$ such that $\delta=\delta_{_{d}}$, for some metric $d$ on $X$, that is, $(X,\delta)$ is a metric approach space \cite[Chapter 3]{MR1472024}. \vspace{0.4cm}

\noindent(4) {\it $\Equ$ to $\App$-$\Equ$.} Each equilogical space $\left\langle (X,\tau),\equiv_{_{X}}\right\rangle$ becomes an approach equilogical one $\left\langle (X,\delta_{_{\tau}}),\equiv_{_{X}}\right\rangle$ with the approach distance given by 
\begin{align*}\delta_{_{\tau}}(x',A)=\left\{\begin{array}{rl} 0, & \text{if $A\in\x$, for some $\x\in UX$ with $\x\rightarrow x'$} \\ \infty, & \text{otherwise}, \end{array}\right.\end{align*} 
for each $x'\in X$, $A\in\mathcal{P}X$, where $UX$ denotes the set of ultrafilters on $X$ \cite[III-Examples 2.4.1(2)]{MR3307673}. The left adjoint of this embedding is slightly more elaborate: for an approach equilogical object $\left\langle (X,\delta),\equiv_{_{X}}\right\rangle$, consider the convergence relation $\rightarrow$ between ultrafilters in $UX$ and points of $X$ given by 
\begin{align*}\x\rightarrow x \ \Longleftrightarrow \ {\rm sup}\{\delta(x,A) \ | \ A\in\x\}<\infty;\end{align*} 
this convergence defines a pseudo-topological space \cite{MR0025716}, to which we apply the reflector described in \cite[III-Exercise 3.D]{MR3307673}, obtaining an equilogical space $\left\langle (X,\tau_{_{\delta}}),\equiv_{_{X}}\right\rangle$. Hence $\Equ$ is identified within $\App$-$\Equ$ as the approach equilogical objects $\left\langle (X,\delta),\equiv_{_{X}}\right\rangle$ such that $\delta=\delta_{_{\tau}}$, for some topology $\tau$ on $X$, that is, $(X,\delta)$ is a topological approach space \cite[Chapter 2]{MR1472024}. 

\vspace{0.5cm}\noindent{\bf Open question.} The conditions {\bf (a)} to {\bf (f)} of Sections \ref{sec2} and \ref{sec3} were derived from the successful attempt of generalizing the structures/constructions to $\Cat{(\T,\V)}$, for suitable $\T$ and $\V$. Requiring those conditions on an arbitrary category with a topological functor over $\Set$ produced the same desired results. However, we do not know an example of a category satisfying those conditions which cannot be described as $\Cat{(\T,\V)}$. 

\section*{Acknowledgments}
This work was done during the preparation of the author's PhD thesis, under the supervision of Maria Manuel Clementino, whom the author thanks for proposing the investigation and advising the whole study. I also thank Fernando Lucatelli Nunes for fruitful discussions.


\begin{thebibliography}{GHK+80}

\bibitem[AHS90]{MR1051419}
Ji\v{r}\'i Ad\'amek, Horst Herrlich, and George~E. Strecker.
\newblock {\em Abstract and concrete categories}.
\newblock Pure and Applied Mathematics (New York). John Wiley \& Sons, Inc.,
  New York, 1990.
\newblock The joy of cats, A Wiley-Interscience Publication.

\bibitem[BBS04]{MR2072989}
Andrej Bauer, Lars Birkedal, and Dana~S. Scott.
\newblock Equilogical spaces.
\newblock {\em Theoret. Comput. Sci.}, 315(1):35--59, 2004.

\bibitem[BCRS98]{MR1659606}
Lars Birkedal, Aurelio Carboni, Giuseppe Rosolini, and Dana~S. Scott.
\newblock Type theory via exact categories (extended abstract).
\newblock In {\em Thirteenth {A}nnual {IEEE} {S}ymposium on {L}ogic in
  {C}omputer {S}cience ({I}ndianapolis, {IN}, 1998)}, pages 188--198. IEEE
  Computer Soc., Los Alamitos, CA, 1998.

\bibitem[Car95]{MR1358759}
Aurelio Carboni.
\newblock Some free constructions in realizability and proof theory.
\newblock {\em J. Pure Appl. Algebra}, 103(2):117--148, 1995.

\bibitem[CCH15]{MR3416170}
Dimitri Chikhladze, Maria~Manuel Clementino, and Dirk Hofmann.
\newblock Representable {$(\mathbb{T},{\bf V})$}-categories.
\newblock {\em Appl. Categ. Structures}, 23(6):829--858, 2015.

\bibitem[CH03]{MR1990036}
Maria~Manuel Clementino and Dirk Hofmann.
\newblock Topological features of lax algebras.
\newblock {\em Appl. Categ. Structures}, 11(3):267--286, 2003.

\bibitem[CH04]{MR2116322}
Maria~Manuel Clementino and Dirk Hofmann.
\newblock On extensions of lax monads.
\newblock {\em Theory Appl. Categ.}, 13:No. 3, 41--60, 2004.

\bibitem[CH09]{MR2491799}
Maria~Manuel Clementino and Dirk Hofmann.
\newblock Lawvere completeness in topology.
\newblock {\em Appl. Categ. Structures}, 17(2):175--210, 2009.

\bibitem[CHJ14]{MR3175322}
Maria~Manuel Clementino, Dirk Hofmann, and George Janelidze.
\newblock The monads of classical algebra are seldom weakly {C}artesian.
\newblock {\em J. Homotopy Relat. Struct.}, 9(1):175--197, 2014.

\bibitem[Cho48]{MR0025716}
Gustave Choquet.
\newblock Convergences.
\newblock {\em Ann. Univ. Grenoble. Sect. Sci. Math. Phys. (N.S.)}, 23:57--112,
  1948.

\bibitem[CHR18]{CHRCCECT}
Maria~Manuel Clementino, Dirk Hofmann, and Willian Ribeiro.
\newblock Cartesian closed exact completions in topology.
\newblock Preprint 18-46, Dept. Mathematics, Univ. Coimbra, arXiv 1811.03993,
  2018.

\bibitem[CM82]{MR678508}
Aurelio Carboni and R.~Celia Magno.
\newblock The free exact category on a left exact one.
\newblock {\em J. Austral. Math. Soc. Ser. A}, 33(3):295--301, 1982.

\bibitem[CR00]{MR1787592}
Aurelio Carboni and Giuseppe Rosolini.
\newblock Locally {C}artesian closed exact completions.
\newblock {\em J. Pure Appl. Algebra}, 154(1-3):103--116, 2000.
\newblock Category theory and its applications (Montreal, QC, 1997).

\bibitem[CT03]{MR1957813}
Maria~Manuel Clementino and Walter Tholen.
\newblock Metric, topology and multicategory---a common approach.
\newblock {\em J. Pure Appl. Algebra}, 179(1-2):13--47, 2003.

\bibitem[CT14]{MR3330902}
Maria~Manuel Clementino and Walter Tholen.
\newblock From lax monad extensions to topological theories.
\newblock In {\em Categorical methods in algebra and topology}, volume~46 of
  {\em Textos Mat./Math. Texts}, pages 99--123. Univ. Coimbra, Coimbra, 2014.

\bibitem[CV98]{MR1600009}
Aurelio Carboni and Enrico~M. Vitale.
\newblock Regular and exact completions.
\newblock {\em J. Pure Appl. Algebra}, 125(1-3):79--116, 1998.

\bibitem[CVO17]{MR3731477}
Eva Colebunders and Karen Van~Opdenbosch.
\newblock Topological properties of non-{A}rchimedean approach spaces.
\newblock {\em Theory Appl. Categ.}, 32:Paper No. 41, 1454--1484, 2017.

\bibitem[Day72]{MR0296126}
Brian Day.
\newblock A reflection theorem for closed categories.
\newblock {\em J. Pure Appl. Algebra}, 2(1):1--11, 1972.

\bibitem[GHK+80]{MR614752}
Gerhard Gierz, Karl~Heinrich Hofmann, Klaus Keimel, Jimmie~D. Lawson,
  Michael~W. Mislove, and Dana~S. Scott.
\newblock {\em A compendium of continuous lattices}.
\newblock Springer-Verlag, Berlin-New York, 1980.

\bibitem[Hof07]{MR2355608}
Dirk Hofmann.
\newblock Topological theories and closed objects.
\newblock {\em Adv. Math.}, 215(2):789--824, 2007.

\bibitem[Hof11]{MR2729224}
Dirk Hofmann.
\newblock Injective spaces via adjunction.
\newblock {\em J. Pure Appl. Algebra}, 215(3):283--302, 2011.

\bibitem[Hof14]{MR3227304}
Dirk Hofmann.
\newblock The enriched {V}ietoris monad on representable spaces.
\newblock {\em J. Pure Appl. Algebra}, 218(12):2274--2318, 2014.

\bibitem[HR13]{MR2981702}
Dirk Hofmann and Carla~D. Reis.
\newblock Probabilistic metric spaces as enriched categories.
\newblock {\em Fuzzy Sets and Systems}, 210:1--21, 2013.

\bibitem[HST14]{MR3307673}
Dirk Hofmann, Gavin~J. Seal, and Walter Tholen, editors.
\newblock {\em Monoidal topology}, volume 153 of {\em Encyclopedia of
  Mathematics and its Applications}.
\newblock Cambridge University Press, Cambridge, 2014.
\newblock A categorical approach to order, metric, and topology.

\bibitem[Law02]{MR1925933}
F.~William Lawvere.
\newblock Metric spaces, generalized logic, and closed categories [{R}end.
  {S}em. {M}at. {F}is. {M}ilano {\bf 43} (1973), 135--166 (1974); {MR}0352214
  (50 \#4701)], 2002.
\newblock With an author commentary: Enriched categories in the logic of
  geometry and analysis.

\bibitem[Low97]{MR1472024}
Robert Lowen.
\newblock {\em Approach spaces}.
\newblock Oxford Mathematical Monographs. The Clarendon Press, Oxford
  University Press, New York, 1997.
\newblock The missing link in the topology-uniformity-metric triad, Oxford
  Science Publications.

\bibitem[Man02]{MR1902090}
Ernest~G. Manes.
\newblock Taut monads and {$T0$}-spaces.
\newblock {\em Theoret. Comput. Sci.}, 275(1-2):79--109, 2002.

\bibitem[Men00]{ecatmm}
Mat\'ias Menni.
\newblock {\em Exact completions and toposes}.
\newblock PhD thesis, University of Edinburgh, 2000.

\bibitem[MST06]{MR2248546}
Mojgan Mahmoudi, Christoph Schubert, and Walter Tholen.
\newblock Universality of coproducts in categories of lax algebras.
\newblock {\em Appl. Categ. Structures}, 14(3):243--249, 2006.

\bibitem[Ros98]{esfsgr}
Giuseppe Rosolini.
\newblock Equilogical spaces and filter spaces, manuscript, 1998.

\bibitem[Ros99]{MR1721095}
Ji\v{r}\'{i} Rosick\'y.
\newblock Cartesian closed exact completions.
\newblock {\em J. Pure Appl. Algebra}, 142(3):261--270, 1999.

\bibitem[Sch84]{MR785032}
Friedhelm Schwarz.
\newblock Product compatible reflectors and exponentiability.
\newblock In {\em Categorical topology ({T}oledo, {O}hio, 1983)}, volume~5 of
  {\em Sigma Ser. Pure Math.}, pages 505--522. Heldermann, Berlin, 1984.

\bibitem[Sco96]{ANCDSER}
Dana~S. Scott.
\newblock A new category? domains, spaces and equivalence relations,
  manuscript, 1996.

\end{thebibliography}
\end{document}